\documentclass[12pt]{amsart}
\usepackage{amsmath,amssymb,amsbsy,amsfonts,amsthm,latexsym,mathabx,
            amsopn,amstext,amsxtra,euscript,amscd,stmaryrd,mathrsfs,
            cite,array,mathtools,enumerate}
            
\usepackage{url}
\usepackage[colorlinks,linkcolor=blue,anchorcolor=blue,citecolor=blue]{hyperref}
\usepackage{color}
\begin{document}

\newtheorem{theorem}{Theorem}
\newtheorem{lemma}[theorem]{Lemma}
\newtheorem{claim}[theorem]{Claim}
\newtheorem{cor}[theorem]{Corollary}
\newtheorem{prop}[theorem]{Proposition}
\newtheorem{definition}{Definition}
\newtheorem{question}[theorem]{Question}
\newcommand{\hh}{{{\mathrm h}}}

\numberwithin{equation}{section}
\numberwithin{theorem}{section}

\def\sssum{\mathop{\sum\!\sum\!\sum}}
\def\ssum{\mathop{\sum\ldots \sum}}

\def \balpha{\boldsymbol\alpha}
\def \bbeta{\boldsymbol\beta}
\def \bgamma{{\boldsymbol\gamma}}
\def \bomega{\boldsymbol\omega}

\newcommand{\Res}{\mathrm{Res}\,}
\newcommand{\Gal}{\mathrm{Gal}\,}

\def\sssum{\mathop{\sum\!\sum\!\sum}}
\def\ssum{\mathop{\sum\ldots \sum}}
\def\dsum{\mathop{\sum\  \sum}}
\def\iint{\mathop{\int\ldots \int}}

\def\squareforqed{\hbox{\rlap{$\sqcap$}$\sqcup$}}
\def\qed{\ifmmode\squareforqed\else{\unskip\nobreak\hfil
\penalty50\hskip1em\null\nobreak\hfil\squareforqed
\parfillskip=0pt\finalhyphendemerits=0\endgraf}\fi}

\newfont{\teneufm}{eufm10}
\newfont{\seveneufm}{eufm7}
\newfont{\fiveeufm}{eufm5}
%
%
\newfam\eufmfam
     \textfont\eufmfam=\teneufm
\scriptfont\eufmfam=\seveneufm
     \scriptscriptfont\eufmfam=\fiveeufm
%
%
\def\frak#1{{\fam\eufmfam\relax#1}}

\def\fK{\mathfrak K}
\def\fT{\mathfrak{T}}

\def\fA{{\mathfrak A}}
\def\fB{{\mathfrak B}}
\def\fC{{\mathfrak C}}
\def\fD{{\mathfrak D}}

\newcommand{\sX}{\ensuremath{\mathscr{X}}}

\def\eqref#1{(\ref{#1})}

\def\vec#1{\mathbf{#1}}
\def\dist{\mathrm{dist}}
\def\vol#1{\mathrm{vol}\,{#1}}

\def\squareforqed{\hbox{\rlap{$\sqcap$}$\sqcup$}}
\def\qed{\ifmmode\squareforqed\else{\unskip\nobreak\hfil
\penalty50\hskip1em\null\nobreak\hfil\squareforqed
\parfillskip=0pt\finalhyphendemerits=0\endgraf}\fi}

\def\sA{\mathscr A}
\def\sB{\mathscr B}
\def\sC{\mathscr C}
\def\sD{\Delta}
\def\sE{\mathscr E}
\def\sF{\mathscr F}
\def\sG{\mathscr G}
\def\sH{\mathscr H}
\def\sI{\mathscr I}
\def\sJ{\mathscr J}
\def\sK{\mathscr K}
\def\sL{\mathscr L}
\def\sM{\mathscr M}
\def\sN{\mathscr N}
\def\sO{\mathscr O}
\def\sP{\mathscr P}
\def\sQ{\mathscr Q}
\def\sR{\mathscr R}
\def\sS{\mathscr S}
\def\sU{\mathscr U}
\def\sT{\mathscr T}
\def\sV{\mathscr V}
\def\sW{\mathscr W}
\def\sX{\mathscr X}
\def\sY{\mathscr Y}
\def\sZ{\mathscr Z}

\def\cA{{\mathcal A}}
\def\cB{{\mathcal B}}
\def\cC{{\mathcal C}}
\def\cD{{\mathcal D}}
\def\cE{{\mathcal E}}
\def\cF{{\mathcal F}}
\def\cG{{\mathcal G}}
\def\cH{{\mathcal H}}
\def\cI{{\mathcal I}}
\def\cJ{{\mathcal J}}
\def\cK{{\mathcal K}}
\def\cL{{\mathcal L}}
\def\cM{{\mathcal M}}
\def\cN{{\mathcal N}}
\def\cO{{\mathcal O}}
\def\cP{{\mathcal P}}
\def\cQ{{\mathcal Q}}
\def\cR{{\mathcal R}}
\def\cS{{\mathcal S}}
\def\cT{{\mathcal T}}
\def\cU{{\mathcal U}}
\def\cV{{\mathcal V}}
\def\cW{{\mathcal W}}
\def\cX{{\mathcal X}}
\def\cY{{\mathcal Y}}
\def\cZ{{\mathcal Z}}
\newcommand{\rmod}[1]{\: \mbox{mod} \: #1}

\def\vr{\mathbf r}

\def\e{{\mathbf{\,e}}}
\def\ep{{\mathbf{\,e}}_p}
\def\em{{\mathbf{\,e}}_m}
\def\en{{\mathbf{\,e}}_n}

\def\Tr{{\mathrm{Tr}}}
\def\Nm{{\mathrm{Nm}}}
\def\supp{{\mathrm{supp}}}

 \def\SS{{\mathbf{S}}}

\def\lcm{{\mathrm{lcm}}}

\def\({\left(}
\def\){\right)}
\def\fl#1{\left\lfloor#1\right\rfloor}
\def\rf#1{\left\lceil#1\right\rceil}

\def\mand{\qquad \mbox{and} \qquad}

\newcommand{\commA}[1]{\marginpar{%
\begin{color}{red}
\vskip-\baselineskip 
\raggedright\footnotesize
\itshape\hrule \smallskip A: #1\par\smallskip\hrule\end{color}}}

\newcommand{\commI}[1]{\marginpar{%
\begin{color}{blue}
\vskip-\baselineskip 
\raggedright\footnotesize
\itshape\hrule \smallskip I: #1\par\smallskip\hrule\end{color}}}




\hyphenation{re-pub-lished}
\hyphenation{ne-ce-ssa-ry}

\parskip 4pt plus 2pt minus 2pt

\def\bfdefault{b}
\overfullrule=5pt

\def \F{{\mathbb F}}
\def \K{{\mathbb K}}
\def \L{{\mathbb L}}
\def \Z{{\mathbb Z}}
\def \Q{{\mathbb Q}}
\def \R{{\mathbb R}}
\def \C{{\mathbb C}}
\def\Fp{\F_p}
\def \fp{\Fp^*}

\title{Sums of inverses in thin sets of finite fields}

\author{Igor E. Shparlinski and Ana Zumalac\'{a}rregui}
\thanks{This work was supported by   ARC Grant DP140100118}

\address{Department of Pure Mathematics, University of New South Wales\\
2052 NSW, Australia.}

\email{igor.shparlinski@unsw.edu.au}
\email{a.zumalacarregui@unsw.edu.au}

\begin{abstract} We obtain lower bounds for the cardinality of $k$-fold 
sum-sets of reciprocals of elements of suitable defined short intervals in high degree extensions of finite fields. Combining our results with bounds for multilinear character sums we obtain new results on incomplete multilinear Kloosterman sums in finite fields.
\end{abstract}


\maketitle

\section{Introduction}

\subsection{Background}

Let $p$ be a prime number and let $\F_p$ be the finite field of $p$ elements. 

Bourgain and Garaev~\cite{BG} have studied the additive properties of the multiple sum-sets 
of reciprocals from a short interval, that is, sum-sets  of 
 $$
\cI_{u,h}^{-1}=\left\{x^{-1}~:~ x\in \cI_{u,h},\, x\ne 0\right\},
$$
where  $\cI_{u,h}$ is the reduction modulo $p$ of the set of consecutive integers
 $\{u+1, \ldots, u+h\}$ for some integers $h$ and $u$ with $p > h\ge 1$. 
In particular, by~\cite[Theorem~4]{BG} 
there is an absolute constant $c>0$ such that
if $h\le p^{c/k^2}$, then for the $k$ folded sum-set of $\cI_{u,h}^{-1}$, that is, for
$$
k\(\cI_{u,h}^{-1}\)=\left \{x_1^{-1}+\cdots+x_k^{-1}~:~x_j\in \cI_{u,h},\ x_j\ne 0, i =1, \ldots, k\right \}, 
$$
for any fixed $k$ and $h \to \infty$, we have
\begin{equation}\label{k(I)}
\# k\(\cI_{u,h}^{-1}\) \ge h^{k+o(1)}.
\end{equation}
Observe that this estimate is almost optimal, since we trivially have that $\# k\(\cI_{u,h}^{-1}\)\le h^k$.

Here we study an analogous problem in large extensions of finite fields. Namely, 
let $\F_q$ be the finite field of $q$ elements, of characteristic $p$, and let $\overline{\F}_q$ be the algebraic closure of $\F_q$.  We fix an algebraic element $\alpha\in \overline{\F}_q$ of degree $n$ over 
$\F_q$, and denote by $\psi(x)$ its characteristic polynomial (that is, $\psi\in\F_q[T]$ is 
a monic irreducible polynomial of degree $n$ and $\psi(\alpha)=0$). 
Then we have that the finite extension $\F_{q^n}$ is isomorphic to $\F_q[\alpha]$, or equivalently $\F_{q^n}\cong\F_q[T]/\psi(T)$.  In this setting the natural generalisation of a \textit{short interval} is the shifted set of \textit{polynomials 
of small degree}.

More precisely, for a given $m\le n$ let us consider the following vector space of dimension $m$
\[
\cV_m=\left\{x~:~x= a_0+a_1\alpha+\cdots+ a_{m-1}\alpha^{m-1}, 
\ a_0, \ldots, a_{m-1} \in\F_q\right\}.
\]
Note that every element $x$ in $\cV_m$ can be identified with a polynomial $x(T)$ of degree at most 
$m-1$ in $\F_q[T]/\psi(T)$.

For a fixed element $\gamma\in\F_{q^n}$, we are interested in the additive properties of the inverses 
of elements in the affine vector space $\cJ_{\gamma,m}= \{\gamma\}+ \cV_m$. 
It is also convenient to define 
$$
 \cJ_{\gamma,m}^* = \cJ_{\gamma,m}\setminus\{0\},
$$
which we call an {\it interval\/} in $\F_{q^n}$. In particular $\cV_m$ plays the role of the {\it initial interval\/}. 

\subsection{Sums of reciprocals from a short interval}

We are interested in counting the number of solutions to 
\[
\frac{1}{x_1+\gamma}+\ldots+\frac{1}{x_k+\gamma} = \frac{1}{x_{k+1}+\gamma}+\ldots+\frac{1}{x_{2k}+\gamma}, 
\]
with  $x_1, \dots, x_{2k} \in  \cV_m$ and fixed $\gamma \in \F_{q^n}$. 

It is clear from the isomorphism $\F_{q^n}\cong\F_q[T]/\psi(T)$ that this problem is equivalent to counting the number  $N_k(\gamma, m,\psi)$ of solutions to 
\begin{align*} \frac{1}{x_1(T)+\gamma(T)}&+\ldots+\frac{1}{x_k(T)+\gamma(T)} \\
&\equiv  \frac{1}{x_{k+1}(T)+\gamma(T)}+\ldots+\frac{1}{x_{2k}(T)+\gamma(T)} \pmod {\psi(T)}
\end{align*}
where $x_1,\dots, x_{2k} \in\F_q[T]$, with $\deg_T(x_i)\le m-1$ for $i=1, \ldots, 2k$,
and a fixed polynomial $\gamma(T) \in\F_q[T]$.

\begin{theorem}\label{thm:k(J)}
Uniformly over $q$, $\gamma\in\F_{q}[T]$ and fixed $k\ge 1$, if  $m< n/(4k^2-2k)$, then 
$$
N_k(\gamma,m,\psi)\le q^{(k+o(1))m}
$$ 
as $m \to \infty$.
\end{theorem}

More concretely, this result gives an equivalent of the bound~\eqref{k(I)} for the set 
$$
 k \(\cJ_{\gamma,m}^{-1} \) = \left\{ x_1^{-1}+\cdots+x_k^{-1}~:~ x_1, \ldots, x_k \in\cJ_{\gamma,m}^*\right\}.
$$
Observe that the trivial bound in this case is 
$$
\# k \(\cJ_{\gamma,m}^{-1} \) \le \(\# \cJ_{\gamma,m}\)^k= q^{mk}.
$$
Now, using the Cauchy inequality we immediately derive Theorem~\ref{thm:k(J)}
(see a short standard proof in Section~\ref{sec:cor1.2}

\begin{cor}\label{cor:k(J)}
Uniformly over $q$, $\gamma\in\F_{q^n}$ and fixed $k\ge 1$, if  $m< n/(4k^2-2k)$, then
$$
\#k \(\cJ_{\gamma,m}^{-1} \) \ge q^{(k+o(1))m}
$$
as $m \to \infty$.
\end{cor}

\subsection{Bounds of character sums}

As it has been noticed by Karatsuba~\cite{K}, see also~\cite{H-B,Pierce},  bounds on the number of solutions of 
equations with reciprocals can be translated into bounds for short multiple Kloosterman sums. In particular, we obtain  an analogue of 
a similar result of Bourgain and Garaev~\cite[Theorem~12]{BG}. We note however that our estimate is weaker
since the underlying tool, the bound on multilinear additive character sums in arbitrary finite fields, due to 
Bourgain and Glibichuk~\cite[Theorem~4]{BGl} is weaker than its counterpart over prime fields 
given by Bourgain~\cite[Theorem~3]{Bou} (but is somewhat more explicit, similarly to~\cite[Theorem~5]{Bou}).

Besides, in the case of  arbitrary finite fields $\F_{q^n}$ there are some necessary restrictions on the size of the intersections of the sets involved with proper subfields of $\F_{q^n}$.  Within the above approach, these sets
are related to the initial data in a rather complicated way so to avoid this difficulty we impose the primality 
condition on both $q$ and $n$.  These conditions can be relaxed, but they allow us to exhibit the ideas in the simplest form.

\begin{theorem}\label{thm:Kloosterman}
Let $\F_{p^n}$ be a finite field with a fixed prime $p$ and  a sufficiently large prime $n$.  
Assume that  positive integers $m$ and $d$ satisfy
$$
m < n/4 \mand d\ge \frac{302900 n}{(m^{1/2}n^{1/2} -2m)}.
$$ 
There   exists $\delta >0$ depending only on $d$, such that
 for any intervals $\cJ_1,\ldots, \cJ_d\subseteq \F_{p^n}$ 
 of dimension $m$,  an additive character $\chi$ in $\F_{p^n}$ 
 and complex weights $\alpha_i(x_i)$ defined on $x_i \in \cJ_i$ with
 $$
 |\alpha_i(x_i)|\le 1, \qquad x_i \in \cJ_i, 
 $$
for $i=1, \ldots, d$, we have
$$
\left |\sum_{x_1\in \cJ_1}\cdots \sum_{x_d\in \cJ_d} \alpha_1(x_1)\cdots\alpha_d(x_d)\,\chi\(\(x_1 \cdots x_d\)^{-1}\)\right |\le p^{dm -\delta n}.
$$
\end{theorem}

\section{Preliminary results}

\subsection{Some general results}

Since we have polynomials in variables $T$ and also in $Z$, to avoid any confusion 
we always write $\deg_T$ to denote the degree in $T$ (even when $Z$ is not present). 

The following result is necessary for the proof of Theorem~\ref{thm:k(J)} and can be found in~\cite[Corollary~3]{Sh}.

\begin{lemma}\label{lem:Res} Let $2\le s,\ell\le k$ be fixed integers. Let $f(Z)$ and $g(Z)$ be polynomials
\[
f(Z)=\sum_{i=1}^{s-1} a_i Z^i \quad \text{ and }\quad g(Z)=\sum_{i=1}^{\ell-1} b_i Z^i,
\]
with polynomial coefficients $a_i(T), b_i(T)\in\F_q[T]$, such that $a_{s-1},b_{\ell-1}\neq 0$ and 
\[\deg_T a_i, \deg_Tb_i < (k-i)M,\qquad  i=1,\ldots,k-1,\]
for some integer $M\ge 1$. Then, the degree of the resultant $\Res(P,Q)$ of $P$ and $Q$ satisfies that
\[
\deg_T \Res(f,g) \le (k^2-1)M.
\]
\end{lemma}

The following result can be found in \cite[Lemma 1]{Sh} and it is useful in the proof of Theorem~\ref{thm:k(J)}.

\begin{lemma}\label{lem:divisors}
The number of divisors of a polynomial $f\in\F_q[T]$ of degree $s$ is $q^{c_0s/\log s}$ for some absolute constant $c_0$.
\end{lemma}

The following result is a modification of~\cite[Lemma~5]{BG}, but the proof is completely analogous and therefore it is omitted.

\begin{lemma}\label{lem:counting}
 let $S$ be a finite subset of a field $K$,  $c\in K$ and  $c_1,\ldots,c_k\in K^*$. 
  Let $T_r$ denote the number of solutions of the equation 
\[c_1x_1+\cdots +c_rx_r=c,\quad x_1,\ldots, x_r\in S,\]
and $J_{2s}$ the number of solutions of the equation
\[x_1+\cdots + x_s=x_{s+1}+\cdots +x_{2s},\quad x_1,\ldots, x_{2s}\in S.\]
If $r=2k$ for some integer $k$, then
$T_{r}\le J_{2k}$. If $r=2k+1$ for some integer $k$, then $T_r^2 \le J_{2k-2}J_{2k}$.
\end{lemma}

Note that Lemma~\ref{lem:counting} is used to exclude degenerate cases when counting solutions to our equation.

The following result can be found in~\cite[Theorem 4]{BGl} and it is a generalization 
of~\cite[ Lemma~1]{BG} for finite fields.

We define $\omega=156450$ (the constant that appears in the formulation 
of Theorem~\ref{thm:Kloosterman}). 

\begin{lemma}\label{lem:entropy_BG} 
Let $r$ be a sufficinetly large prime power and let $d$ be any integer with $3\le d\le 0.9 \log_2 \log_2 r$.
  For $0<\eta\le 1$ define $\tau = \min (1/\omega, \eta/120)$. Suppose that the sets 
  $\cA_1,\cdots, \cA_d\subseteq \F_r^{*}$, with at least 3 elements  are such 
  that for every $i=3,\ldots, d$, for any element $t\in  \F_r^{*}$ and   proper 
  subfield $\L\subseteq \F_r^{*}$ we have $\#(\cA_i\cap t\L)\le \#\cA_i^{1-\eta}$. 
  Assume further that for some $\varepsilon>0$
\[\#\cA_1\cdot \#\cA_2\(\#\cA_3\cdots  \#\cA_d\)^{\tau}>r^{1+\varepsilon}.\]
Then, for any nontrivial additive character $\chi$ of $\F_r$ we have 
\[\left|\sum_{a_1\in \cA_1}\cdots \sum_{a_d\in \cA_d}\chi(a_1\cdots a_d)\right| <100\,  \#\cA_1\cdots \#\cA_d\cdot  r^{-0.45\varepsilon / 2^d}.\]
\end{lemma}

In order to effectively apply Lemma~\ref{lem:entropy_BG} one has to study carefully the elements of $k(\cJ_{\gamma,m}^{-1})$ in subfields. We restrict the study to the simplest case, where there is only one proper subfield.

\begin{cor}\label{cor:entropy_BG} 
Let $\cA_1,\cdots, \cA_d\subseteq \F_{p^n}^{*}$, with $p$ and $n$ prime, of cardinality $\#\cA_i \ge p^{\sigma}$,   $i=1,\ldots, d$, for some real $\sigma$ with 
$$
\frac{\omega n}{2\omega +d-2}<\sigma\le n
$$
Then, there exist $\delta>0$ that depends only on $d$ and $\sigma$ such that for sufficiently large $n$, we have  
\[\left|\sum_{a_1\in \cA_1}\cdots \sum_{a_d\in \cA_d}\chi(a_1\cdots a_d)\right| <100\,  \#\cA_1\cdots \#\cA_d\cdot  p^{-n\delta}.\]
\end{cor}

\begin{proof}
The only proper subfield in $\F_{p^n}$ is $\F_p$, therefore it is clear that 
$$\#(\cA_i\cap t\F_p) \le p   \le p^{\sigma(1-\eta)} \le \#\cA_i^{1-\eta}$$
for any $0<\eta < 1-1/\sigma$. Furthermore, from the hypothesis 
$$\#\cA_1\cdot \#\cA_2\(\#\cA_3\cdots  \#\cA_d\)^{1/\omega}\ge p^{\sigma(2+(d-2)/\omega)}>p^{n(1+\varepsilon)},
$$
for any 
$$
0<\varepsilon< \frac{\sigma(2\omega +d-2)}{\omega n} -1.
$$ 
The result now follows from Lemma~\ref{lem:entropy_BG} for $\delta =0.45\varepsilon / 2^d$.
\end{proof}

\subsection{Sums of inverses in  function fields}

We now establish an analogue of~\cite[Lemma~6]{BG}, which is the main ingredient 
in the proof of Theorem~\ref{thm:k(J)} and   which we believe is of independent interest. 

Let  $\cP_m$ be the set of  polynomials $x\in\F_q[T]$ of degree $\deg_T x <  m$. 
In particular $\# \cP_m = q^m$.

For a fixed $\beta \in \overline{\F_q(T)}$, where $\overline{K}$ denotes the algebraic closure of the field $K$, and positive integers $k$ and $m$ we
 now define 
 $N_k(\beta,m)$ as the number of solutions to the equation
\begin{equation}\label{eq:xb}
\begin{split}
\frac{1}{x_1(T)+\beta(T)}&+\cdots+\frac{1}{x_k(T)+\beta(T)} \\& 
\quad =  \frac{1}{x_{k+1}(T)+\beta(T)}+\cdots+\frac{1}{x_{2k}(T)+\beta(T)}, 
\end{split}
\end{equation}
with $x_i \in \cP_m$ for $i=1,\ldots, 2k$.

\begin{lemma}\label{lem:invers F[T]}
Let $k\ge 1$, uniformly over $q$ and $\beta \in \overline{\F_q(T)}$  we have 
$$
N_k(\beta,m) \le  q^{(k+o(1))m}
$$
as $m \to \infty$.
\end{lemma}

\begin{proof}  To simplify the counting, we split the total number of solutions
$N_k(\beta,m)$
separating the contribution $N_k^{=}(\beta,m)$  from the solutions satisfying $x_i=x_j$ for some $i\ne j$, and the contribution $N_k^{\ne}(\beta,m)$ from the solutions
\begin{equation}
\label{eq:GoodSols}
\vec{x}=(x_1,\ldots,x_{2k})\in \cP_m^{2k}, \quad \text{ with }x_i\ne x_j \text{ for }1 \le i\ne j \le 2k.
\end{equation}

We derive this bound by induction on $k$. 
It is clear that if $k=1$ the assertion is trivial. Suppose that $k\ge 2$.
It follows from Lemma~\ref{lem:counting} that the number $N_k^{=}(\beta,m)$ 
of solutions satisfying $x_i=x_j$ for some $i\ne j$ contributes to the total number of solutions $N_k(\beta,m)$
at most 
$$
N_k^{=}(\beta,m) = O\(\sqrt{N_{k-1}(\beta,m)N_{k}(\beta,m)}+q^mN_{k-1}(\beta,m)\)
$$
(note that the first term is responsible for coincidences between the variables on the same side of the equation, while the second term comes from  coincidences between on opposite sides). 

Hence, by the induction hypothesis
\begin{align*}
N_k^{=}(\beta,m) & = O\(q^{(k/2+o(1))m}\sqrt{N_{k}(\beta,m)}+q^{(k+o(1))m}\)\\
& = O\(q^{(k/2+o(1))m}\sqrt{N_{k}^{=}(\beta,m) + N_k^{\ne}(\beta,m)}+q^{(k+o(1))m}\). 
\end{align*}
Clearly it suffices to show that  $N_k^{\ne}(\beta,m) \le q^{(k+o(1))m}$.

We must now count the number of solutions in $\overline{\F_q(T)}$ to
\begin{equation}\label{eq:beta}
\prod_{i\ne 1}(x_i+\beta)+\cdots +\prod_{i\ne k}(x_i+\beta) = \prod_{i\ne k+1}(x_i+\beta)+\cdots +\prod_{i\ne 2k}(x_i+\beta)
\end{equation} 
with  $x_i\in\F_q[T]$, $x_i\ne x_j$ and $\deg_T(x_i)<m$, for  $1 \le i\ne j \le 2k$.

Observe that if follows from~\eqref{eq:beta} that $\beta$ is algebraic over $\F_q(T)$ of degree $d$, with $1\le d\le 2k-2$. Therefore, it can be written as $\beta =\xi/\mu$ with $\mu \in\F_q[T]$ a polynomial of degree at most $m(2k-2)=O(m)$ and $\xi$ an algebraic integer of degree $d$ over $\F_q[T]$. 

From~\eqref{eq:beta}, the polynomial
\begin{align*}
q^{2k-1}\prod_{j\neq i} (x_j-x_i)& = \prod_{j\neq i}\((\mu x_j +\xi)-(\mu x_i+\xi)\)\\
&= (\mu x_i+\xi)\cdot H(x_1,\ldots,x_{2k})+ \prod_{j\neq i}(\mu x_j +\xi)
\end{align*}
(for some polynomial $H$ in $2k$ variables)
is divisible by $(\mu x_i+\xi)$ in a certain algebraic extension 
of the function field $\F_q(T)$ and is non-zero since $x_i\ne x_j$ for every $i\ne j$.

In particular the norm 
\begin{equation}\label{eq:division}
\Nm(\mu x_i+\xi) \mid q^{(2k-1)d}\prod_{j\ne i}(x_j-x_i)^d 
\end{equation} 
as a polynomial in $\F_q[T]$.

We  now fix $x_1$. Recalling~\eqref{eq:division}, we see that we can decompose the polynomial $\Nm(\mu x_1+\xi)=F_1 G_1$ with $G_1 \mid \mu$ and $\mathrm{gcd}(F_1,q)=1$, also 
$$
 \deg_T F_1,\deg_T G_1 \le d(2k-1)m = O(m).
 $$ 

For every divisor $f_1=\prod_{j\ge 2}r_j$ of $F_1$, with $F_1 \mid f_1^d$, we can construct the arithmetic progressions $\cL_{2,j}$
\begin{equation}\label{eq:AP}
x_j\equiv x_1 \pmod{r_j},\quad 2\le j\le 2k.
\end{equation}
Since the $\deg_T(x_j-x_1)<m$ the number of elements in $\cL_{2,j}$ is at most $q^{m-\deg r_j }$ and therefore
\begin{equation}\label{eq:prod_AP}
\prod_{j\ge 2}\# \cL_{2,j} <  q^{m-\deg r_2} \cdots q^{m-\deg  r_{2k}} = \frac{q^{m(2k-1)}}{q^{\deg f_1}}.
\end{equation} 
For any of the $q^m$ possibilities for $x_1$ it follows from Lemma~\ref{lem:divisors} that there are at most $q^{o(m)}$ choices for $f_1$ and thus, from~\eqref{eq:prod_AP},
\begin{equation}\label{eq:Bound_m1}
N^{\ne}_k(\beta,m)\le\frac{q^{2km+o(m)}}{q^{m_1}}, 
\end{equation}
where $m_1$ is the most popular degree amongst all the polynomials $f_1$.

On the next step for every $x_2\in \cL_{2,2}$ we can factor the norm $\Nm(\mu x_2+\xi)$ as
\[\Nm(\mu x_2+\xi)=F_2G_2\]
where the irreducible factors of $G_2$ either divide $\mu$ or $\Nm(\mu x_1+\xi)$, and $F_2$ is not only coprime with $\mu$ but also with $\Nm(\xi+\mu x_1)$ (and in particular with $(x_2-x_1)$). Once again, from~\eqref{eq:division} it follows that 
$$
F_2 \mid \prod_{j\ge 3}(x_j-x_2)^d.
$$ 
Once again, every divisor $f_2$ of $F_2$ with $F_2\mid f_2^d$ can be written as
$$f_2=\prod_{j\ge 3}s_j,
$$ 
with $s_j |(x_j-x_2)$. 
For every such divisor we can construct arithmetic progressions $\cL_{3,j}\subseteq \cL_{2,j}$, satisfying both~\eqref{eq:AP} and
\[x_j\equiv x_2 \pmod{s_j}, \quad 3\le j\le 2k,\]
where $\gcd(r_j,s_j)=1$. In particular, this implies that 
$$\# \cL_{3,j}\le q^{-\deg s_j} \# \cL_{2,j}
$$ 
and thus
\begin{equation}\label{eq:prod_AP2}
\prod_{j\ge 3}\#\cL_{2,j}\le \frac{1}{q^{\deg f_2}}\prod_{j\ge 3}\#\cL_{3,j}.
\end{equation}
As before, if we denote by $m_2$ the most popular degree amongst all the polynomials $f_2$ it follows from~\eqref{eq:Bound_m1} and~\eqref{eq:prod_AP2} 
\begin{equation*}\label{eq:bound}
N_k^{\ne}(\beta,m) \le \frac{q^{2km+o(m)}}{q^{m_1}q^{m_2}}.
\end{equation*}
For every $x_3\in \cL_{3,3}$ once again we can factorise $\Nm(\mu x_3+\xi)=F_3G_3$ where the irreducible factors of $G_3$ either divide $\mu$, $\Nm(\mu x_1+\xi)$ or $\Nm(\mu x_2+\xi)$ and $F_3$ is coprime to them. For each divisor $f_3$ of $F_3$, with $F_3\mid f_3^d$, we can define arithmetic progressions $\cL_{4,j}\subset \cL_{3,j}$ for $4\le j\le 2k$ as we did before so
\[\prod_{j\ge 4}\#\cL_{3,j}\le \frac{1}{q^{\deg f_3}}\prod_{j\ge 4}\#\cL_{4,j}.\]
The process is now clear: we subsequently fix $x_1 \in \cP_m$, then $x_2\in \cL_{2,2}$, 
then $x_3\in \cL_{3,3}$, and so on,  and estimate the number of solutions as
\begin{equation}\label{eq:bound N_k}
N_k^{\ne}(\beta,m) \le \frac{q^{2km+o(m)}}{q^{m_1+\cdots +m_{2k-1}}}.
\end{equation}
On the other hand, since $G_1\mid \mu^d$ and $\deg \mu^d =O(m)$, it is clear that $\deg G_1 =O(m)$. It follows from Lemma~\ref{lem:divisors} that the number of possibilities for $G_1$ is, independently of $x_1$, at most $q^{o(m)}$. 

Since $F_1 \mid f_1^d$ there are at most $q^{o(m)}$ possibilities for $F_1$ once  $f_1$ is fixed. Thus, for any given $f_1$ there exist at most $q^{o(m)}$ possible values for $x_1$. Taking into account that the degree $m_1$ is chosen so the number of possible polynomials $f_1$ with degree different from $m_1$ is $O(mq^{m_1})$ (that is, it is the most popular degree amongst the all possible choices) we have that there are at most
$q^{m_1+o(m)}$
possible values for $x_1$. For any fixed $x_1$, following the same arguments, we have that the number of possibilities for $x_2$ is at most $q^{m_2+o(m)}$. Thus, continuing this procedure 
\[N_k^{\ne}(\beta,m) \le q^{m_1+\cdots +m_{2k-1}+o(1)}.\]
Combining this bound with~\eqref{eq:bound N_k}, the result follows.
\end{proof}

\section{Proofs on main results}

\subsection{Proof of Theorem~\ref{thm:k(J)}}

Within the proof, we consider the elements in $\F_{q^n}$  as classes of polynomials  $x(T)$ in $\F_q[T]/\psi(T)$, where $\psi(T)$ is an irreducible polynomial of degree $n$. Elements $x_i\in \cV_m$ can be identified precisely with polynomials in $\cP_m\subseteq \F_q[T]$. 

Let us denote by $N_k(\gamma,m,\psi)$ the number of solutions to
\begin{equation}\label{eq:xb2}
\sum_{i=1}^k\frac{1}{x_i(T)+\gamma(T)} \equiv \sum_{j={k+1}}^{2k}\frac{1}{x_{j}(T)+\gamma(T)}\pmod{\psi(T)}, 
\end{equation}
with  $x_1,\ldots,x_{2k}\in\F_q[x]$ with $\deg_T(x_i)<m$, and   $N_k^{\neq}(\gamma,m,\psi)$ the number of solutions with $x_i\ne x_j$ for $1 \le i\ne j \le 2k$. As in the proof of Lemma 3.1 
it suffices to show that $N_k^{\neq}(\gamma,m,\psi)\le q^{(k+o(1))m}$.

For every  solution $\vec{x}=(x_1,\ldots,x_{2k})$ contributing to $N_k^{\neq}(\gamma,m,\psi)$ we  construct the polynomial
\begin{equation}
\begin{split}\label{def:P_x}
P_{\vec{x}}(Z)=& \sum_{s=1}^k\prod_{j\ne s} (x_j(T)+Z)- \sum_{s=k+1}^{2k}\prod_{i\ne s}^k (x_i
(T)+Z)\\
=&\  A_0(T)+A_1(T)Z+\cdots A_{2k-2}(T)Z^{2k-2}\in\F_q[T][Z].
\end{split}
\end{equation}
It follows from~\eqref{eq:xb2} that $P_{\vec{x}}(\gamma(T))\equiv 0 \pmod{\psi(T)}$. Furthermore, since by hypothesis $x_1(T)\ne x_i(T)$ for $i=2,\ldots, 2k$, it is clear that
\[P_{\vec{x}}(-x_1(T))=\prod_{i\ne 1}(x_i(T)-x_1(T))\neq 0,\]
and the polynomial is non-constant in $\F_q[T][Z]$. In fact, the polynomial is non-constant modulo $\psi(T)$ either, 
since 
$$\deg_T P_{\vec{x}}(-x_1) < (2k-1)m < n=\deg_T\psi
$$ 
by hypothesis.

Observe that since for every $1\le i\le 2k$ we have  $\deg_T x_i<m$,   the coefficients of $P_{\vec{x}}$ satisfy
$$
\deg_T A_j<(2k-1-j)m,\quad\text{ for } j=0,\ldots,2k-2.
$$

For any $\vec{x},\vec{y}$ two solutions the corresponding polynomials $P_{\vec{x}},P_{\vec{y}}$ satisfy: $P_{\vec{x}}(\gamma)\equiv P_{\vec{y}}(\gamma)\equiv 0 \pmod{\psi(T)}$ 
and hence its resultant
\[\Res(P_{\vec{x}},P_{\vec{y}})\equiv 0 \pmod{\psi(T)}. \]
Furthermore, it follows from Lemma~\ref{lem:Res} that 
\begin{equation*}\label{eq:degRes}
\deg_T  \Res(P_{\vec{x}},P_{\vec{y}})\le 4k(k-1)m<n=\deg_T(\psi)
\end{equation*}
so in fact 
$
\Res(P_{\vec{x}},P_{\vec{y}})=0
$
as a polynomial with coefficients in $\F_q[T]$. In particular, this implies that any two polynomials $P_{\vec{x}}, P_{\vec{y}}$ have a common root in $\overline{\F}_q(T)$. 

We fix a solution $\vec{c}=(c_1,\ldots,c_{2k})$ and  consider the set $\{\beta_1,\ldots,\beta_s\}$ of all $s \le 2k-2$ roots of $P_{\vec{c}}(Z)$ in $\overline{\F}_q(T)$,
Then, for every solution $\vec{x}$ the polynomial $P_{\vec{x}}$ has a common root with $P_{\vec{c}}$. 
Hence, the number of solutions to~\eqref{eq:xb2} can be bounded by
$$
 (2k-2) \max_{1\le i\le s}\#\{P\in\F_q[T][Z]:\text{ of the form~\eqref{def:P_x}} \text{ and } P(\beta_i)=0\}.
$$
For any fixed root $\beta\in \{\beta_1,\ldots, \beta_s\}$ of $P_{\vec{c}}$ the number of polynomials of the form~\eqref{def:P_x} with $P(\beta)=0$ is precisely the number $N_k(\beta,m)$ of solutions to~\eqref{eq:xb} which, from Lemma~\ref{lem:invers F[T]}, is at most $q^{(k+o(1))m}$. 

\subsection{Proof of Corollary~\ref{cor:k(J)}}
\label{sec:cor1.2}

Let $T(\lambda)$ be the number of solutions to the equation 
$$
 x_1^{-1}+\cdots+x_k^{-1} = \lambda, \qquad x_1, \ldots, x_k \in\cJ_{\gamma,m}^*. 
$$
Obviously 
$$
\sum_{\lambda\in k \(\cJ_{\gamma,m}^{-1} \)} T(\lambda) = \(\# \(\cJ_{\gamma,m}^* \)\)^{2k}
\ge \(q^m-1\)^{2k} 
$$
and
$$
\sum_{\lambda\in k \(\cJ_{\gamma,m}^{-1} \)} T(\lambda)^2 = N_k(\gamma,m,\psi).
$$
Then, by the Cauchy inequality, we have
\begin{align*}
\(q^m-1\)^{2k} & \le 
\(\# \(\cJ_{\gamma,m}^* \)\)^{2k}\\ &
\le \#k(\cJ_{\gamma,m}^{-1}) \sum_{\lambda\in k \(\cJ_{\gamma,m}^{-1} \)} T(\lambda)^2
=\#k(\cJ_{\gamma,m}^{-1}) N_k(\gamma,m,\psi).
\end{align*}
Using Theorem~\ref{thm:k(J)}, we conclude the proof.

\subsection{Proof of Theorem~\ref{thm:Kloosterman}}

For a non-trivial additive character $\chi$, let
\[S= \sum_{x_1\in \cJ_1}\cdots \sum_{x_d\in \cJ_d} \alpha_1(x_1)\cdots\alpha_d(x_d)\,\chi\(\(x_1 \cdots x_d\)^{-1}\).\]
It follows from  the simple observation that for any complex number $|z|^2 = z\cdot\overline{z}$, that for any sets $\cU, \cV \subseteq \F_{q^n}$  and the weights $\{\alpha(v)\}_{v \in \cV}$ with $|\alpha(v)|\le 1$, we have
\begin{equation}\label{eq:bound_Sd}
\sum_{u \in \cU}\left|\sum_{v \in \cV}\alpha(v) \chi(uv)\right|^{2k} \le \sum_{v_1,\ldots,v_{2k} \in \cV}\left|
\sum_{u \in \cU}\chi\(u \sum_{i=1}^{2k} (-1)^i v_i \)\right|,
\end{equation}
for every integer $k$. 

Let us denote $J=\#\cJ_i= q^{m}$ , to simplify the notation. 

The bound~\eqref{eq:bound_Sd}, together with 
the  H\"{o}lder inequality, applied  $d$ times, exactly as in the proof of~\cite[Theorem~12]{BG}, gives
\begin{align*}
|S|^{(2k)^d}&\le J^{d(2k)^d-2kd}\!\! \sum_{\substack{x_{i,1}\in \cJ_{1}\\
i=1,\cdots,2k}} \!\cdots\!\sum_{\substack{x_{i,d,}\in \cJ_{d}\\
i=1,\cdots,2k}} \prod_{j=1}^d\chi\(\sum_{i=1}^{2k} (-1)^i x_{i,j}^{-1}  \).
\end{align*}
We can fix $x_{2i-1,j}$ for every $i=1,\ldots, k$ and $j=1,\ldots, d$ in such a way that for some elements $c_1, \ldots, c_d$ we have
\begin{align}\label{eq:bound_Holder}\nonumber
|S|^{(2k)^d}&\le J^{d(2k)^d-kd}\  \left|\sum_{\substack{x_{i,1}\in \cJ_{1}\\
i=1,\cdots,k}} \!\cdots\!\sum_{\substack{x_{i,d}\in \cJ_{d}\\
i=1,\cdots,k}}\chi\(\prod_{j=1}^{d} 
\(\sum_{i=1}^{k}  x_{i,j}^{-1}-c_j\)\)\right|\\
&\le  J^{d(2k)^d -kd} \,\left|\sum_{\lambda_1 \in \F_{q^n}}\cdots \sum_{\lambda_d \in \F_{q^n}}{T_1(\lambda_1)}\cdots T_d(\lambda_d)\, \chi(\lambda_1\cdots \lambda_d)\right|,
\end{align}
where $T_j(\lambda)$ denotes the number of solutions to
$$
 y_{1}^{-1}+\cdots + y_{k}^{-1}-c_j = \lambda, \qquad y_1, \ldots, y_k \in \cJ_{j}.
$$
For any $k$, to be chosen later, satisfying
\begin{equation}\label{eq:choice_k}
m < \frac{n}{4k^2-2k}
\end{equation}
we have from Theorem~\ref{thm:k(J)} that the number of solutions to the congruence
$$
 y_{1}^{-1}+\cdots + y_{k}^{-1}=  y_{k+1}^{-1}+\cdots+ y_{2k}^{-1}, \qquad  y_1, \ldots, y_{2k}\in \cJ_j, 
$$
is bounded by $J^{k+o(1)}$.
Therefore, in particular, for every $j=1,\ldots,d$, we have 
\begin{equation}\label{eq:2norm}
\|T_j\|_2^2 = \sum_{\lambda\in \F_{q^n}}{T_j(\lambda)^2} \le J^{k+o(1)}.
\end{equation}
Now, let us estimate the sum in~\eqref{eq:bound_Holder}, that is, 
\[W =  \sum_{\lambda_1 \in \F_{q^n}}\cdots \sum_{\lambda_{d-1} \in \F_{q^n}}{ T_1(\lambda_1)}\cdots {T_d(\lambda_d)}\, \chi(\lambda_1\cdots \lambda_{d-1}) .\]
Let  
\[
\cA_j=\{\lambda \in \F_{q^n}~:~\lambda =  y_{1}^{-1}+\cdots + y_{k}^{-1}-c_j,\text{ for }y_1, \ldots, y_k \in \cJ_{j}\}
\]
be the set on which $T_j(\lambda)$ is supported, $j=1,\ldots, d$.
By the Cauchy inequality 
\begin{align*} 
 |W|^2 &\le 
 \left\|{T_1}\right\|_2^2\cdots \left\|{T_{d-1}}\right\|_2^2 \\
 & \qquad \left|\sum_{u,v\in   \cA_d}
\sum_{\lambda_1 \in\cA_1}\cdots \sum_{\lambda_{d-1} \in\cA_{d-1}}  {T_d(u)T_d(v)} \chi(\lambda_1\cdots \lambda_{d-1}(u-v))\right| .
\end{align*} 
Hence, using   the bounds in~\eqref{eq:2norm}, we have
\begin{equation} 
\begin{split}
\label{eq:W} 
 |W|^2  & \le J^{(d-1)k+o(1)}  \sum_{u,v \in \cA_d} \\
 & \qquad \quad\left| \sum_{\lambda_1 \in\cA_1}\cdots \sum_{\lambda_{d-1} \in\cA_{d-1}} {T_d(u)T_d(v)} 
 \chi\(\lambda_1\cdots \lambda_{d-1}(u-v)\)\right|. 
\end{split}
\end{equation}
Let us note that, to estimate the contribution from the inner  sum in~\eqref{eq:W} we use Corollary~\ref{cor:entropy_BG} with $d-1$ and $\sigma=m(k+o(1))$. Clearly, from Corollary~\ref{cor:k(J)},
which applies due to the condition~\eqref{eq:choice_k}, 
we have $\#\cA_j  = J^{k+o(1)}$, $ j =1, \ldots, d$. Also, let us further assume that $k$ satisfies
\begin{equation}\label{eq:choice_k2} 
m> \frac{\omega n}{k(2 \omega+ d-3)}.
\end{equation}
Therefore, it follows from Corollary~\ref{cor:entropy_BG} that the  contribution to~\eqref{eq:W} from diagonal terms with $u=v$ is precisely
\begin{equation}\label{eq:diagonal}
\left\|{T_d}\right\|_2^2 \#\cA_1\cdots \#\cA_{d-1}\le J^{dk+o(1)}.
\end{equation}
 It follows from Corollary~\ref{cor:entropy_BG} that for some $\delta_0>0$, depending only on $d$,    for every   $t\in  \F_{q^n}^*$ we have
\begin{equation}\label{eq:entropy}
\left|\sum_{\lambda_1\in \cA_1}\cdots\sum_{\lambda_{d-1}\in \cA_{d-1}}\chi(\lambda_1\cdots \lambda_{d-1}t)\right| \le J^{(d-1)k}p^{-\delta_0 n}.
\end{equation}
This bound, together with the  trivial observation 
$$\sum_{u\in \cA_d}{T_d(u)} = J^k,
$$
 implies that the contribution to~\eqref{eq:W} from non-diagonal terms with $u \ne v$ is at most  $J^{(d+1)k}q^{-\delta_0 n}$.

Combining this bound with~\eqref{eq:diagonal} and~\eqref{eq:W} we have that for any 
choice of $k$ satisfying~\eqref{eq:choice_k} and~\eqref{eq:choice_k2}
\[W^2 \le \(J^{-k} + p^{-\delta_0 n}\)J^{2kd+o(1)},\]
together with~\eqref{eq:bound_Holder} implies that 
\[|S| \le J^{d+o(1)} \(J^{-k} +p^{-\delta_0 n} \)^{1/2(2k)^d} \le p^{dm-\delta n}\]
for some positive $\delta$.
To complete the proof it suffices to choose any integer $k$ satisfying the conditions~\eqref{eq:choice_k} 
and~\eqref{eq:choice_k2}.  In fact it is slightly more convenient to work with a slightly more stringent condition than~\eqref{eq:choice_k2} and   we choose $k$ to satisfy
\[\frac{\omega\cdot n}{(d+2\omega -3)m}<
\frac{\omega n }{d m}<
k <\frac{1}{2}\left(\frac{n}{m}\right)^{1/2}.\]
In particular, the existence of such a $k\in \mathbb N$ in the previous range can be guaranteed if 
\[ 
\frac{1}{2}\left(\frac{n}{m}\right)^{1/2} - \frac{\omega n}{dm} \ge 1,
\] 
or equivalently
\[d\ge \frac{2\omega n}{(m^{1/2}n^{1/2} -2m)},\]
which   coincides with the hypothesis for $d$, $m$ and $n$.

\section{Comments}

 We   note that is not difficult to get 
an explicit (but rather cluttered)  expression for the saving $\delta$ in Theorem~\ref{thm:Kloosterman}.

Although we have presented extensions to finite fields of only two selected results 
of Bourgain and Garaev~\cite{BG}, one can easily check that our approach allows to 
get extensions of several other bounds from~\cite{BG}. However, these methods do not 
apply to yet another natural generalisation of~\cite{BG} when the role of short intervals 
is played by arbitrary  low-dimensional affine vector  subspaces over $\F_{q^n}$ (considered as a vector space 
over $\F_q$) rather than by very special  affine  spaces $\cJ_{\gamma,m}$.   We recall that 
for additive character sums with polynomials such results are known, see~\cite{Ost}. However 
 character sums with rational functions, even in the simplest case of multilinear sums with reciprocals,
are not covered by this technique and seem to require new ideas.
We also remark that some of the motivation to question come from the problem of constructing efficient 
affine dispersers and extractors over finite fields, see~\cite{B-SG,B-SK} for more details
and further references.

\end{document}